\documentclass[preprint,pra,superscriptaddress]{revtex4-1}

\usepackage[dvips]{graphicx} 
\usepackage{amsfonts}
\usepackage{amssymb}
\usepackage{amscd}
\usepackage{amsmath}    
\usepackage{amsthm}   
\usepackage{enumerate}
\usepackage{epsfig}
\usepackage{subfigure}

\newtheorem{theorem}{Theorem}[section]
\newtheorem{lemma}[theorem]{Lemma}
\newtheorem{proposition}[theorem]{Proposition}
\newtheorem{corollary}[theorem]{Corollary}
\newtheorem{definition}[theorem]{Definition}

\begin{document}

\title{Explicit combinatorial design}

\author{Xiongfeng Ma}
\email{xma@tsinghua.edu.cn}
\affiliation{Center for Quantum Information, Institute for Interdisciplinary Information Sciences, Tsinghua University, Beijing, China}
\affiliation{%
Center for Quantum Information and Quantum Control,\\
Department of Physics and Department of Electrical \& Computer Engineering,\\
University of Toronto, Toronto,  Ontario, Canada}

\author{Zhen Zhang}
\email{zzhang12@mails.tsinghua.edu.cn}
\affiliation{Center for Quantum Information, Institute for Interdisciplinary Information Sciences, Tsinghua University, Beijing, China}

\author{Xiaoqing Tan}
\email{ttanxq@jnu.edu.cn}
\affiliation{%
Department of  Mathematics, College of Information Science and Technology, \\
Jinan University, Guangzhou, Guangdong, P.~R.~China}
\affiliation{%
Center for Quantum Information and Quantum Control,\\
Department of Physics and Department of Electrical \& Computer Engineering,\\
University of Toronto, Toronto,  Ontario, Canada}

\begin{abstract}
A combinatorial design is a family of sets that are almost disjoint, which is applied in pseudo random number generations and randomness extractions. The parameter, $\rho$, quantifying the overlap between the sets within the family, is directly related to the length of a random seed needed and the efficiency of an extractor. Nisan and Wigderson proposed an explicit construction of designs in 1994. Later in 2003, Hartman and Raz proved a bound of $\rho\le e^2$ for the Nisan-Wigderson construction in a limited parameter regime. In this work, we prove a tighter bound of $\rho<e$ with the entire parameter range by slightly refining the Nisan-Wigderson construction. Following the block idea used by Raz, Reingold, and Vadhan, we present an explicit weak design with $\rho=1$.
\end{abstract}

\maketitle

\section{Introduction}
Combinatorial designs play an important role in pseudo random number generations \cite{Nisan:Design:1994} and randomness extractions \cite{Trevisan:Extractor:1999}. Nisan and Wigderson propose a simple construction of designs (Nisan-Wigderson design) for pseudo random number generators \cite{Nisan:Design:1994}, which is later applied to construct randomness extractors by Trevisan \cite{Trevisan:Extractor:1999}.

A combinatorial design is a family of subsets, drawn from the set, which have a same size, $q$, and are almost disjoint. For a family of disjoint subsets, the size of the set, $l$, grows linearly with the number of subsets, $n$. Later, we will see that with a design, the size of the set only grows as $poly(\log{n})$.

One key parameter of a design, $\rho$, is used to quantify the overlap between subsets in the family. Generally speaking, the smaller $\rho$ is, the more disjoint the subsets are. This parameter is linked to the seed length and approximately indicates the ratio of randomness that can be extracted by Trevisan's extractor \cite{Trevisan:Extractor:1999,Raz:Extractor:2002}. In the application of extractors, $\rho$ is normally required to be close to 1. Furthermore, the size of the set, $l$, is linked to the initial randomness input (as seed) required for Trevisan's extractor. In general, the size ($l$) should be small compared to the number of subsets ($n$).

Hartman and Raz proved a bound of $\rho\le e^2$ ($e$ as the Euler's number) for the Nisan-Wigderson design \cite{Hartman:Design:2003} when $n$ is a power of a prime power number, $q$ (subset size). By slightly refining the Nisan-Wigderson design, we prove a better bound $\rho<e$ for the entire range of $n\le q^q$. Furthermore, we follow the block idea used by Raz, Reingold, and Vadhan to construct an explicit design with $\rho=1$ and $l=O(\log^3 n)$.

In Section \ref{Section:Design:Preliminaries}, we review the definitions of combinatorial designs, the Nisan-Wigderson design and the Hartman-Raz bound. In Section \ref{Section:Design:Newbound}, we refine the Nisan-Wigderson design and show a better bound of $\rho$. In Section \ref{Section:Design:Construction}, we construct an explicit $\rho=1$ design. We finally conclude with discussions in Section \ref{Section:Design:Conclusion}.

%
%
%
%

\section{Preliminaries} \label{Section:Design:Preliminaries}

\subsection{Notations and Definitions} \label{Sub:Design:Def}
Notations: $[l]=\{0,1,2,\dots,l-1\}$; $\log$ is base 2; $\ln$ is the natural logarithm; and $e$ is the base of the natural logarithm or the Euler's number.

Define a Galois (finite) field, $GF(q)=[q]$ where $q=p^r$, $r$ is a positive integer, and $p$ is a prime. Here, we represent an element, $j\in[q]$, by a $p$-nary string. Define $\mathcal{F}_q$ to be the ring of polynomials over the field $GF(q)$. For a polynomial $\phi(x)\in\mathcal{F}_q$, denote $\lambda(\phi)$ to be its number of roots over $GF(q)$. For the sake of simplicity, we use $p=2$ in the following. We remark that our results apply to the case of a general prime $p$ with minor modifications.

Denote $\mathcal{M}_{q^{d+1}}=\{\phi_0,\phi_1,\dots,\phi_{q^{d+1}-1}\}\subseteq \mathcal{F}_q$ to be the set of all polynomials over $GF(q)$ with the highest order no greater than $d\in[q]$, and hence, $|\mathcal{M}_{q^{d+1}}| = q^{d+1}$. We further divide the set $\mathcal{M}_{q^{d+1}}$ evenly into $q$ disjoint subsets, $\mathcal{N}_{d,j}$ with $j\in GF(q)$,
\begin{equation} \label{Design:Define:Ndj}
\begin{aligned}
\mathcal{N}_{d,j} \triangleq \{ j x^d+\phi(x) | \phi(x)\in \mathcal{M}_{q^{d}}\}.
\end{aligned}
\end{equation}
That is, the coefficient of $x^d$ of each polynomial in $\mathcal{N}_{d,j}$ is $j$. It is not hard to see that
\begin{equation} \label{Design:PM0i}
\begin{aligned}
\mathcal{M}_{q^{d+1}} &= \bigcup_{j=0}^{q-1}\mathcal{N}_{d,j}, \\
\mathcal{N}_{d,0} &= \mathcal{M}_{q^{d}} \\
\end{aligned}
\end{equation}
and hence for every $j\in[q]$,
\begin{equation} \label{Design:Ndj}
\begin{aligned}
|\mathcal{N}_{d,j}| &= q^d. \\
\end{aligned}
\end{equation}

For a polynomial set, $\mathcal{M}$, define a function,
\begin{equation} \label{Design:Define:Lambda}
\begin{aligned}
\Lambda(\mathcal{M}) \triangleq \sum_{\phi\in\mathcal{M}}2^{\lambda(\phi)}
\end{aligned}
\end{equation}
In the summation on the right side, we assume that the number of roots of the trivial polynomial $\phi=0$ is zero. That is, for every constant function $\phi$,
\begin{equation} \label{Design:Define:phiconstfun}
\begin{aligned}
\lambda(\phi\equiv \text{const})=0.
\end{aligned}
\end{equation}

\subsection{Designs}
A combinatorial design is a family (collection) of nearly disjoint subsets of a set $[l]$. Here are the two definitions of designs used in the literature.

\begin{definition} \label{Def:Ext:Design}
(Standard Design) A family of sets $S_0, S_1,\dots,S_{n-1}\subseteq[l]$ is a standard $(n,q,l,\rho)$-design if
\begin{enumerate}
\item
For all $i \in[n]$, $|S_i|=q$.

\item
For all $i\neq j \in[n]$,
\begin{equation} \label{Design:StandardDef}
\begin{aligned}
|S_i\cap S_j|\le\log\rho.
\end{aligned}
\end{equation}
\end{enumerate}
\end{definition}

\begin{definition} \label{Def:Ext:WeakDesign}
(Weak design) A family of sets $S_0, S_1,\dots,S_{n-1}\subseteq[l]$ is a weak $(n,q,l,\rho)$-design if
\begin{enumerate}
\item
For all $i \in[n]$, $|S_i|=q$.

\item
For all $i \in[n]$,
\begin{equation} \label{Design:WeakDef}
\begin{aligned}
\sum_{j<i}2^{|S_i\cap S_j|} \le n\rho.
\end{aligned}
\end{equation}
\end{enumerate}
\end{definition}

We remark that in this work, we use a slightly stronger version of Eq.~\eqref{Design:WeakDef},
\begin{equation} \label{Design:WeakDefi}
\begin{aligned}
\sum_{j<i}2^{|S_i\cap S_j|} \le (i+1)\rho.
\end{aligned}
\end{equation}
Here, $(i+1)\rho\le n\rho$, since $i\in[n]$.

%

Definition \ref{Def:Ext:Design} is originally used in the Nisan-Wigderson construction \cite{Nisan:Design:1994} that is applied in the Trevisan extractor \cite{Trevisan:Extractor:1999}. Then, Raz et al.~showed that a weaker version of design (Definition \ref{Def:Ext:WeakDesign}) is sufficient for the use in the Trevisan extractor \cite{Raz:Extractor:2002} \footnote{In the original definition of weak design, $(n-1)\rho$ instead of $n\rho$ is used on the right side of Eq.~\eqref{Design:WeakDef}. Here we follow the definition in \cite{Hartman:Design:2003}.}. Later, Hartman and Raz proved a bound of $\rho$ of the Nisan-Wigderson construction for a modified version of the weak design \cite{Hartman:Design:2003}.

A design can be treated as an $l \times n$ binary (or $p$-nary) matrix with the $i$th row represents a subset $S_{i-1}$, for example, $n=4$, $q=2$, $l=4$ and a binary matrix
\begin{equation} \label{Ext:design:44example}
\begin{aligned}
A = \left(
  \begin{array}{cccc}
    1 & 0 & 1 & 0 \\
    0 & 1 & 0 & 1 \\
    1 & 0 & 0 & 1 \\
    0 & 1 & 1 & 0 \\
  \end{array}
\right).
\end{aligned}
\end{equation}
Take $[l]=\{0,1,2,3\}$, then the family of sets are $S_1=\{0,2\}$, $S_2=\{1,3\}$, $S_3=\{0,3\}$ and $S_4=\{1,2\}$. It is not hard to see that $\rho=2$ for the standard design from Eq.~\eqref{Design:StandardDef}, while $\rho=5/4$ in the weak design definition of Eq.~\eqref{Design:WeakDef}.

As pointed earlier, the objective of design construction is to minimize $l$ and $\rho$, given $q$ and $n$. In this work, we will derive a tight upper bound of $\rho$ for (weak) designs from the modified Nisan-Wigderson construction.

\subsection{Nisan-Wigderson design}
Without loss of generality, let the size of set (the length of the random seed in the application of Trevisan's extractor), $l$, be the square of a prime power number ($l=q^2$, if not, pick the smallest power of 2 which is greater than $\sqrt{l}$).
Consider $[l]$ to be a $q\times q$ 2-dimensional array, then every element of $[l]$ can be represented as a pair of elements in $GF(q)$. The Nisan-Wigderson design is constructed as follows.
\begin{enumerate}
\item
Find $n$ distinct polynomials $\{\phi_0(\cdot),\phi_1(\cdot),\dots,\phi_{n-1}(\cdot)\}$ on $GF(q)$ of degree at most $d$. This can be done as long as $n\le q^{d+1}$ and $d\in[q]$.

\item
The nearly disjoint sets are given by
\begin{equation} \label{Design:NWcon}
\begin{aligned}
S_i  = \{ <j,\phi_i(j)>|j\in GF(q)\}
\end{aligned}
\end{equation}
where $<j,\phi_i(j)>$ presents an element in $[l]$.
\end{enumerate}
The following facts can be easily verified \cite{Nisan:Design:1994}:
\begin{enumerate}
\item
The size of each set is exactly $q$, $|S_i|=q$ for every $i\in[q]$.

\item
Any two sets intersect in at most $d$ points. 

\item
There are at least $q^{d+1}$ possible sets (the number of polynomials on $GF(q)$ of degree at most $d$).
\end{enumerate}

In the original proposal of the Nisan-Wigderson design, the polynomials (with a degree at most $d$) are chosen in an arbitrary manner. A natural way to choose these polynomials is to go from low order polynomials to higher ones, which results the highest order of polynomials to be $d=\lceil\log{n}/\log{q}-1\rceil\le\log{n}$. According to Definition \ref{Def:Ext:Design}, it is straightforward to see that $\rho\le\log{n}$ as shown by Nisan and Wigderson \cite{Nisan:Design:1994}.

\subsection{Hartman-Raz bound}
Hartman and Raz proved that the Nisan-Wigderson design is an explicit modified weak $(n,q,l,\rho)$-design with $l=q^2$ and $\rho\le e^2$ in Theorem 1 of ref.~\cite{Hartman:Design:2003}. We remark that Hartman and Raz's result is only proven to for the case when $n$ is a power of $q$.

\section{New bound} \label{Section:Design:Newbound}
Intuitively, the more sets the design has, the harder to make sets disjoint. Thus, one might conjecture that the parameter $\rho$, defined in Eq.~\eqref{Design:WeakDef}, grows with $n$. Mathematically, this is not necessarily true, because the overlap is normalized by $n$, as shown in Eq.~\eqref{Design:WeakDef}. In fact, one can find counter examples to this conjecture for Nisan-Wigderson design. In the following, we present a new design construction by slightly refining the original Nisan-Wigderson design. We show that for any $n\le q^q$, one can obtain the upper bound $\rho<(1+q^{-1})^q$, which shows that the refined Nisan-Wigderson design is an explicit weak $(n,q,l,\rho)$-design with $\rho<e$ (see, Theorem \ref{Theorem:Design:bounde}).

\subsection{Refined Nisan-Wigderson design}
Here, we refine the Nisan-Wigderson design by choosing the $i$th polynomial for Eq.~\eqref{Design:NWcon} in the following manner:
\begin{equation} \label{Design:PolyChoose}
\begin{aligned}
\phi_i(x) 
&= \sum_{k=0}^{d} {(\lfloor i/q^k\rfloor \mod q)} x^k, \\
\end{aligned}
\end{equation}
where $i\in[n]$, $d=\lceil\log{n}/\log{q}-1\rceil$ (then, $q^{d}<n\le q^{d+1}$), and the coefficients calculated by the modulo function ${(\lfloor i/q^k\rfloor \mod q)}$ are elements of $GF(q)$. These polynomials form a set \begin{equation} \label{Design:Define:DefMn}
\begin{aligned}
\mathcal{M}_{n}=\{\phi_0,\phi_1,\dots,\phi_{n-1}\},
\end{aligned}
\end{equation}
and by the definition of Eq.~\eqref{Design:Define:Ndj},
\begin{equation} \label{Design:Define:NncontainMd0}
\begin{aligned}
\mathcal{N}_{d,0} = \{\phi_0,\phi_1,\dots,\phi_{q^d-1}\} \subset \mathcal{M}_{n}.
\end{aligned}
\end{equation}
Each polynomial, $\phi_i$, in $\mathcal{M}_{n}$ corresponds to a set $S_i$ in the design in the form of Eq.~\eqref{Design:NWcon}.

\subsection{Evaluation of $\rho$}
In the following discussion, we evaluate the parameter $\rho$ in Eq.~\eqref{Design:WeakDefi} for the design given by Eq.~\eqref{Design:PolyChoose}.
The number of intersection elements $|S_i\cap S_j|$ equals to the number of roots of $\phi_i=\phi_j$ or
\begin{equation} \label{Design:Modified:IntersectionRoot}
\begin{aligned}
|S_i\cap S_j| = \lambda(\phi_i-\phi_j).
\end{aligned}
\end{equation}
Then, the left hand side of Eq.~\eqref{Design:WeakDefi} can be written as
\begin{equation} \label{Design:Modified:suminter}
\begin{aligned}
\sum_{j<i}2^{|S_i\cap S_j|} = \sum_{j<i}2^{\lambda(\phi_i-\phi_j)}.
\end{aligned}
\end{equation}

\begin{proposition}
For any two sets defined in Eq.~\eqref{Design:Define:Ndj}, $\mathcal{N}_{d,i}$ and $\mathcal{N}_{d,j}$ with $ij\neq0$ and $i,j\in GF(q)$, there exists a one-to-one map between them such that the two polynomials by the map have the same roots.
\end{proposition}
\begin{proof}
The map can be constructed by multiplying a scalar, $i/j \mod q$, to the second set, since $ij\neq0$ and $i/j\in GF(q)$.
\end{proof}
We remark that the two polynomials not only have the same number of roots but also the same values. According to the definition of $\Lambda(\cdot)$, Eq.~\eqref{Design:Define:Lambda}, it is simple to see the following lemma.
\begin{lemma} \label{Lemma:Design:LambdaNdj}
The value of $\Lambda(\mathcal{N}_{d,j})$is the same for all $j\neq0\in GF(q)$.
\end{lemma}

For the case where $j=0$, we have the following lemma.
\begin{lemma} \label{Lemma:Design:SameMd10}
For every positive integer $d$,
\begin{equation} \label{Design:SameMd10}
\begin{aligned}
\Lambda(\mathcal{N}_{d,0})\le\Lambda(\mathcal{N}_{d,1}).
\end{aligned}
\end{equation}
\end{lemma}

\begin{proof}
From Lemma 4 of ref.~\cite{Leontev:Roots:2006}, we know that
\begin{equation} \label{Design:LeontevLemma4}
\begin{aligned}
\Lambda(\mathcal{N}_{d,1}) = |\mathcal{N}_{d,1}|\sum_{i=0}^{d}q^{-i}\binom{q}{i}.
\end{aligned}
\end{equation}
With Eq.~\eqref{Design:Ndj},
\begin{equation} \label{Design:Prove0le1}
\begin{aligned}
\Lambda(\mathcal{N}_{d,0}) &= \sum_{k=0}^{d-1}(q-1)\Lambda(\mathcal{N}_{k,1})+1 \\
&\le \left(\sum_{k=0}^{d-1}(q-1)|\mathcal{N}_{k,1}|+1\right)\sum_{i=0}^{d}q^{-i}\binom{q}{i} \\
&= \Lambda(\mathcal{N}_{d,1}) \\
\end{aligned}
\end{equation}
\end{proof}

\begin{lemma} \label{Lemma:Design:calrho}
For all $i \in[q^q]$,
\begin{equation} \label{Design:Modified:Lambda}
\begin{aligned}
\sum_{j\le i}2^{\lambda(\phi_i-\phi_j)} =\sum_{k=k^*}^d a_k\Lambda(\mathcal{N}_{k,1})-\Lambda (\mathcal{N}_{k^*,1})+\Lambda(\mathcal{N}_{k^*,0}),
\end{aligned}
\end{equation}
where
\begin{equation} \label{Design:Modified:VarDef}
\begin{aligned}
d &= \lfloor\log{(i+1)}/\log{q}\rfloor, \\
a_k &= \lfloor (i+1)/q^k\rfloor \mod q, \\
k^* &= \min_{k}\{a_k\neq0\}. \\
\end{aligned}
\end{equation}
\end{lemma}

\begin{proof}
Divide the summation on the left hand side of Eq.~\eqref{Design:Modified:Lambda} into blocks, according to Eq.~\eqref{Design:Modified:VarDef},
\begin{equation} \label{Design:Modified:i1expand}
\begin{aligned}
i+1 = \sum_{k=0}^{d}a_kq^k. \\
\end{aligned}
\end{equation}
Before we prove the Lemma, let us take a look at the first block. According to the definition of $d$, we know that $a_d\neq0$, thus
\begin{equation} \label{Design:Modified:rho1st}
\begin{aligned}
\sum_{j=0}^{q^d-1}2^{\lambda(\phi_i-\phi_j)} &= \sum_{\phi_j\in \mathcal{N}_{d,0}}2^{\lambda(\phi_i-\phi_j)} \\
&=\sum_{\phi_{j'}\in \mathcal{N}_{d,a_d}}2^{\lambda(\phi_{j'})} \\
&=\Lambda (\mathcal{N}_{d,1}), \\
\end{aligned}
\end{equation}
where the two equalities comes from the definition of $\mathcal{N}_{d,j}$, Eq.~\eqref{Design:Define:Ndj}, and the last equality is derived from Eq.~\eqref{Design:Define:Lambda} and Lemma \ref{Lemma:Design:LambdaNdj}. 

With the construction of Eq.~\eqref{Design:PolyChoose} and the expansion of Eq.~\eqref{Design:Modified:i1expand}, we know that the $x^{d}$ coefficient of $\phi_i$ is $a_d$, and that of $\phi_j$ is in $[a_d]$ for every $j\in[a_dq^d]$. Then, according to Lemma \ref{Lemma:Design:LambdaNdj} and the calculation of Eq.~\eqref{Design:Modified:rho1st}, one can see that
\begin{equation} \label{Design:Modified:rho1storder}
\begin{aligned}
\sum_{j=0}^{a_dq^d-1}2^{\lambda(\phi_i-\phi_j)} &=a_d\Lambda (\mathcal{N}_{d,1}), \\
\end{aligned}
\end{equation}
which is contribution from the first term of Eq.~\eqref{Design:Modified:i1expand}.

Following the derivation of Eq.~\eqref{Design:Modified:rho1storder}, we now consider the general term in Eq.~\eqref{Design:Modified:i1expand}. For every $a_k>0$, $k^* \le k\le d$ and $0 \le c_k< a_k$, define a set
\begin{equation} \label{Design:Modified:blocks_def}
\begin{aligned}
\mathcal{A}_{k,c_k}=\{a_d x^d +\dots+a_{k+1}x^{k+1}+c_k x^k+\phi(x) | \phi(x)\in \mathcal N_{k,0}\}.
\end{aligned}
\end{equation}
It is not hard to see that the polynomial sets, $\mathcal{A}_{k,c_k}$, are disjoint for different values of $k$ and $c_k$, and the $\mathcal{M}_{i+1}$ defined in Eq.~\eqref{Design:Define:DefMn} can be partitioned by
\begin{equation} \label{Design:Modified:DisjointA}
\begin{aligned}
\mathcal{M}_{i+1} = \bigcup_{k=k^*}^{d}\bigcup_{c_k=0}^{a_k-1}\mathcal{A}_{k,c_k},
\end{aligned}
\end{equation}
where we use the fact that $\mathcal{A}_{k,c_k}=\varnothing$ when $a_k=0$.

For the last partition, where $k=k^*$ and $c_{k^*}=a_{k^*}-1$, one can see that
\begin{equation} \label{Design:Modified:blocks_compute2}
\begin{aligned}
\sum_{\phi_j\in \mathcal{A}_{k^*,c_{k^*}}} 2^{\lambda(\phi_i-\phi_j)} &=\sum_{\phi_{j'}\in \mathcal N_{k,0}}2^{\lambda(\phi_{j'})}
&=\Lambda (\mathcal N_{k,0}),\\
\end{aligned}
\end{equation}
For any other partitions,
\begin{equation} \label{Design:Modified:blocks_compute}
\begin{aligned}
\sum_{\phi_j\in \mathcal{A}_{k,c_k}}2^{\lambda(\phi_i-\phi_j)}
&=\sum_{\phi_{j'}\in \mathcal N_{k,1}}2^{\lambda(\phi_{j'})}
&=\Lambda (\mathcal N_{k,1})\\
\end{aligned}
\end{equation}
where the first equalities in  Eq.~\eqref{Design:Modified:blocks_compute2} and Eq.~\eqref{Design:Modified:blocks_compute} come from the fact that the coefficients of the highest $d-k$ orders in $\phi_i$ are the same as the ones in every polynomial $\phi_j$ in $\mathcal{A}_{k^*,c_{k^*}}$ or $\mathcal{A}_{k,c_{k}}$. Now with Eq.~\eqref{Design:Modified:DisjointA}, \eqref{Design:Modified:blocks_compute2}, and \eqref{Design:Modified:blocks_compute}, we can evaluate the left hand side of Eq.~\eqref{Design:Modified:Lambda},
\begin{equation} \label{Design:Modified:blocks_sum}
\begin{aligned}
\sum_{j=0}^{i}2^{\lambda(\phi_i- \phi_j)}&=\sum_{\phi_j\in M_i}2^{\lambda(\phi_i- \phi_j)}\\
&=\sum_{\phi_j \in\bigcup \mathcal{A}_{k,c_k}}2^{\lambda(\phi_i- \phi_j)}\\
&=\sum_{\phi_j\in \mathcal{A}_{k^*,a_{k^*}-1}}2^{\lambda(\phi_i- \phi_j)}+\sum_{\phi_j\in \bigcup \mathcal{A}_{k,c_k}/ \mathcal{A}_{k^*,a_{k^*}-1}}2^{\lambda(\phi_i- \phi_j)}\\
&=\Lambda(\mathcal N_{k_*,0})+\sum_{(k,c_k)\ne(k^*,a_{k^*}-1 )}\Lambda(\mathcal N_{k,1}) \\
&=\sum_{k=k^*}^d a_k\Lambda(N_{k,1})-\Lambda (N_{k^*,1})+\Lambda(N_{k^*,0}). \\
\end{aligned}
\end{equation}
\end{proof}

\subsection{Main result}
\begin{theorem} \label{Theorem:Design:bounde}
For a prime power number $q$ and every positive integer $n\le q^q$, there exists an explicit weak $(n,q,l,\rho)$-design with $l=q^2$ and $\rho<(1+q^{-1})^{q}<e$.
\end{theorem}
\begin{proof}
We prove this theorem by showing that the design constructed by Eq.~\eqref{Design:PolyChoose} is a weak $(n,q,l,\rho)$-design with  $\rho<(1+q^{-1})^{q}$.

From the definition of Eq.~\eqref{Design:WeakDefi} and \eqref{Design:Modified:suminter}, one can see that
\begin{equation} \label{Design:Modified:calrho1}
\begin{aligned}
\rho &= \frac{\sum_{j<i}2^{|S_i\cap S_j|}}{i+1} \\
&= \frac{\sum_{j<i}2^{\lambda(\phi_i-\phi_j)}}{i+1}. \\
\end{aligned}
\end{equation}
We then apply Lemma \ref{Lemma:Design:calrho}, Eq.~\eqref{Design:Modified:Lambda} and \eqref{Design:Modified:i1expand} to evaluate $\rho$,
\begin{equation} \label{Design:Modified:calrho2}
\begin{aligned}
\rho &= \frac{ \sum_{k=0}^d a_k\Lambda(N_{k,1})-\Lambda (N_{k^*,1})+\Lambda(N_{k^*,0})-1}{a_d \times q^d+a_{d-1}\times q^{d-1}+...+a_0},
\end{aligned}
\end{equation}
where the factor $-1$ in the numerator comes from the definition of $\lambda$, Eq.~\eqref{Design:Define:phiconstfun}, regarding the term $2^{\lambda(\phi_i-\phi_i)}=1$. Then, according to Lemma \ref{Lemma:Design:SameMd10},
\begin{equation} \label{Design:Modified:calrho2}
\begin{aligned}
\rho &< \frac{ \sum_{k=0}^d a_k\Lambda(N_{k,1})}{a_d \times q^d+a_{d-1}\times q^{d-1}+...+a_0} \\
&\le \max_{k} \frac{\Lambda(N_{k,1})}{q^k} \\
\end{aligned}
\end{equation}
From Eq.~\eqref{Design:Ndj} and \eqref{Design:LeontevLemma4}, one can show that
\begin{equation} \label{Design:Modified:Proverhoe}
\begin{aligned}
\rho &\le \sum_{j=0}^{d}q^{-j}\binom{q}{j} \\
&\le (1+q^{-1})^q-q^{-q} \\
&< (1+q^{-1})^q, \\
\end{aligned}
\end{equation}
where the second inequality holds when $d=q-1$.

\end{proof}

\section{Design construction} \label{Section:Design:Construction}
In Theorem \ref{Theorem:Design:bounde}, we show that the design constructed by Eq.~\eqref{Design:PolyChoose} can be bounded $\rho<e$. On the other hand, it not hard to see that $\rho>2$ for the refined Nisan-Wigderson design (as constructed by Eq.~\eqref{Design:PolyChoose}) in a reasonable regime of $n$ and $q$, e.g., $q\ge16$ and $n>q^2$. Thus, our bound in Theorem \ref{Theorem:Design:bounde} is relatively tight.

In the application of extractors, such as \cite{Raz:Extractor:2002}, the value of $\rho$ roughly indicates the ratio of randomness that can be extracted. Thus, we need to achieve a $\rho$ that is close to 1. Then, we have to go beyond the Nisan-Wigderson design. In order to reduce the parameter $\rho$, one can extend the size of the set, from $[l]$ to $[l']$. Raz et al.~proposed a block design idea to reduce $\rho$ \cite{Raz:Extractor:2002,Hartman:Design:2003}.  The basic idea is break the set $[l']$ into $b$ blocks (smaller sets), each of which has a size of $l$ (hence, $l'=lb$). That is, the $i$th subset is $\{il+1,il+2,\dots,(i+1)l\}$ and $i\in[b]$. The design sets are subsets of one of subsets. Obviously, the sets from different subsets are disjoint. Hartman and Raz show that with this technique (Lemma 17 of ref.~\cite{Raz:Extractor:2002}), $\rho$ can be reduced to 1 exponentially fast with the number of subsets grows. With this technique, we can reduce $\rho$ down to 1 with a finite number, $O(\rho\log(n\rho))$, of blocks by digging into details of the design constructed by Eq.~\eqref{Design:PolyChoose}.

\begin{corollary} \label{Corollary:Design:BlockDesigns}
Given the explicit weak $(n,q,l,\rho)$-design constructed by Eq.~\eqref{Design:PolyChoose} with $l=q^2$ and $1<\rho<e$, there exists an explicit weak $(n',q,l',1)$-design with $n'=n\rho$, $l'=q^2b$ and
\begin{equation} \label{Design:Block:numBlocks}
\begin{aligned}
b &= \left\lfloor\frac{\log n+\log \rho-\log q}{\log\rho-\log(\rho-1)}\right\rfloor \\
&= O(\log n) \\
\end{aligned}
\end{equation}
as the number of blocks.
\end{corollary}
\begin{proof}
Denote the number of subsets from $i$th subset to be $n_i$. We construct the design in such a way that
\begin{equation} \label{Design:Block:blocksizes}
\begin{aligned}
n_i &= (1-\rho^{-1})^i n \\
n_b &= n\rho-\sum_{i=0}^{b-1}n(1-\rho^{-1})^i \\
&= n\rho(1-\rho^{-1})^b \\
\end{aligned}
\end{equation}
where the first equation holds for $i\in[b]$. It is not hard to verify that $\sum_{i=0}^b n_i=n\rho$ and $n_b\le q$ with Eq.~\eqref{Design:Block:numBlocks}.
Now, we can verify the conditions in Definition \ref{Def:Ext:WeakDesign}. Condition 1 is obviously satisfied. For a set $S_j$ in block $i\in[b]$,
\begin{equation} \label{Design:Block:rhoblocks}
\begin{aligned}
\sum_{j'<j}2^{|S_j\cap S_{j'}|} &\le \sum_{i'=0}^{i-1} n_{i'}+\rho n_i = n\rho, \\
\end{aligned}
\end{equation}
since there is no intersection between any $j'$th set and the set from $i$th block. For the last block, $\sum_{j'<j}2^{|S_j\cap S_{j'}|}=j$. Thus, it is a weak $(n',q,l',1+1/n')$-design. Since $\lfloor(1+1/n')(n'-1)\rfloor=n'-1$, it is also a weak $(n',q,l',1)$-design.
\end{proof}

If we use the matrix representation of designs as shown in Eq.~\eqref{Ext:design:44example}, then the new design matrix from a refined Nisan-Wigderson design matrix $A_0$ can be written as
\begin{equation} \label{Design:Block:matrix}
\begin{aligned}
\left(
  \begin{array}{cccccc}
    A_0 &  \\
     & A_1  \\
     &  & A_2 &  \\
     &  &  & \cdots \\
     &  &  &  & A_{b-1}\\
     &  &  &  &  & A_b\\
  \end{array}
\right),
\end{aligned}
\end{equation}
where all the off-diagonal blocks are 0. According to the block design idea, presented in Corollary \ref{Corollary:Design:BlockDesigns}, $A_i$ take first $n_i$ rows of $A_{i-1}$ for $i=\{1,2,\dots,b\}$, where $n_i$ is defined in Eq.~\eqref{Design:Block:blocksizes}.

\section{Discussions} \label{Section:Design:Conclusion}
In Nisan-Wigderson construction, $n$ is limited by $q^q$, which is not necessarily true for a general case. Let us extend the example of Eq.~\eqref{Ext:design:44example},
\begin{equation} \label{Design:46example}
\begin{aligned}
\left(
  \begin{array}{cccc}
    1 & 0 & 1 & 0 \\
    0 & 1 & 0 & 1 \\
    1 & 0 & 0 & 1 \\
    0 & 1 & 1 & 0 \\
    1 & 1 & 0 & 0 \\
    0 & 0 & 1 & 1 \\
  \end{array}
\right).
\end{aligned}
\end{equation}
One can easily verify that this design has a $\rho<2$ and $n=6>q^q=4$. The key point is that one does not need to pick only one element from one block, as used in Eq.~\eqref{Design:NWcon}. In general, one might expect $n=O(\binom{l}{q})$ or $l=O(\log n)$. If one can find such a design with a reasonable $\rho$, one can apply the block design idea as shown in Eq.~\eqref{Design:Block:numBlocks} so that the seed length for the Trevisan extractor is $O(\log^2n)$ .

\subsection*{Acknowledgments}
We thank H.-K.~Lo, B.~Qi, C.~Rockoff, F.~Xu, and H.~Xu for enlightening discussions. Financial supports from the National Basic Research Program of China Grants No.~2011CBA00300 and No.~2011CBA00301, National Natural Science Foundation of China Grants No.~61073174, No.~61033001, No.~61061130540, and No.~61003258, the 1000 Youth Fellowship program in China, CFI, CIPI, the CRC program, CIFAR, MITACS, NSERC, OIT, QuantumWorks, and Special Funds for Work Safety of Guangdong Province of 2010 from Administration of Work Safety of Guangdong Province of China are gratefully acknowledged. X.~Q.~Tan especially thanks H.-K.~Lo for the hospitality during her stay at the University of Toronto.

\bibliographystyle{apsrev4-1}

\bibliography{Bibli}


\end{document}